\documentclass[11pt,english]{amsart}

\usepackage[english]{babel}
\usepackage[T1]{fontenc}
\usepackage{lmodern}

\theoremstyle{plain}
\newtheorem{theo}{Theorem}[section]
\newtheorem{lemm}[theo]{Lemma}
\newtheorem{prop}[theo]{Proposition}

\theoremstyle{definition}

\theoremstyle{remark}

\usepackage{hyperref}
\usepackage{microtype}
\usepackage{enumerate}
\usepackage{graphicx}
\usepackage{color}
\usepackage{fullpage}
\usepackage{bm}
\usepackage{mathtools}

\usepackage{enumitem}
\usepackage{textcomp}
\usepackage[normalem]{ulem}

\title{Another probabilistic construction of $\phi^{2n}$ in dimension $2$}
\author{Yichao Huang}
\address{University of Helsinki, Department of Mathematics and Statistics, P.O. Box 68, FIN-00014 University of Helsinki, Finland}
\email{yichao.huang@helsinki.fi}
\date{}

\begin{document}

\maketitle

\section{Introduction}

The main input of this note is to provide an alternative probabilistic approach to the $\phi^{2n}$ theory in dimension $2$, based on concentration phenomenon of martingales associated to polynomials of Gaussian variables. This is based on an adaptation of the work \cite{lacoin2018path} of Lacoin-Rhodes-Vargas, in which exponential potentials associated to quantum Mabuchi $K$-energy are studied.

We give an alternative proof of the following classical result.

\begin{theo}[Negative exponential moments]\label{th:Main}
Let $n\geq 2$ be an integer and let $R$ be a real, unitary polynomial of even degree $2n$. Let $X$ be the (Dirichlet) Gaussian Free Field on a bounded simply connected domain $\Lambda\subset\mathbb{R}^2$.

Consider the (non-necessary positive) Wick-ordered random measure
\begin{equation}
V_{R}(\Lambda)=\int_{\Lambda}:R(X)(x):d^2x
\end{equation}
with integer $n\geq 2$. Then we have the following estimate
\begin{equation}\label{eq:ExponentialBound}
\mathbb{E}\left[e^{-\alpha V_{R}(\Lambda)}\right]<\infty
\end{equation}
for some $\alpha>0$.
\end{theo}

This key estimate for the construction of the $\phi^{2n}$ theory (where $R(X)=X^{2n}$) in dimension $2$ follows originally from a hypercontractivity argument due to Nelson \cite{nelson1966quartic}. Given this estimate the rest of the argument is standard: the book \cite{reed2012methods} is a good reference for details and developments of the hypercontractivity argument. 

The idea of the martingale method is originally used to study more involved models such as the quantum Mabuchi $K$-energy \cite{lacoin2018path} or the Sine-Gordon model \cite{lacoin2019probabilistic}. This note shows in particular that this idea can also be sucessfully implemented to the Euclidean quantum $\phi^{2n}$ theory in dimension $2$.

We stress that the purpose of this note is to introduce a new and arguably convenient construction of a classical theory in an elementary fashion. Readers unfamiliar with the classical model can consult \cite{simon2015p} for an overview on this subject.

\subsection*{Acknowledgement}
The author is indebted to Christophe Garban, R\'emi Rhodes and Vincent Vargas for communicating the idea. We also acknowledge support from the ERC grant QFPROBA.

\section{Preliminaries}\label{sec:Preliminaries}
\subsubsection*{Notations} In the following we denote by $\Lambda\subset\mathbb{R}^2$ some bounded simply connected open subset of the Euclidean plane $\mathbb{R}^2$. We consider $n\geq 2$ be an integer and let $R$ be a real, unitary polynomial of even degree $2n$. We use $X$ to denote the (Dirichlet) Gaussian Free Field (GFF in short) supported on $\Lambda$. The object of interest would be the Wick-ordered polynomial $:R(X):$ for the GFF. More precisely, we are interested in integrals of type $\int_{\Lambda} :R(X)(x): d^2x$.

\subsection{Gaussian Free Field}
We review some of the aspects of the probabilistic construction of the Gaussian Free Field (or GFF after) that will be useful later. We refer to \cite{dubedat2009sle} for more information.

Recall that the Green function $K(x,y)$ on the domain $\Lambda$ is defined as $K=(-\Delta_\Lambda)^{-1}$, where $-\Delta_\Lambda$ is the differential operator with Dirichlet boundary condition $g=0$ on $\partial\Lambda$. In the following we will stick to the Dirichlet boundary condition although the argument works for general boundary conditions.

A (Dirichlet) GFF $X$ on $\Lambda$ is a random distribution taking value in the negative Sobolev space $H^{-s}(\Lambda)$ with $s>0$. It is characterized by its mean and covariance kernel $K$ on $\Lambda$: for test functions $f,g\in H^{s}(\Lambda)$,
\begin{equation*}
\mathbb{E}\left[\langle X,f\rangle\right]=0,\quad \mathbb{E}\left[\langle X,f\rangle\langle X,g\rangle\right]=\int_{\Lambda^2}f(x)g(y)K(x,y)d^2xd^2y
\end{equation*}
where $\langle,\rangle$ denotes the dual bracket between $H^{-s}(\Lambda)$ and $H^{s}(\Lambda)$. Recall that the Green function $K$ displays logarithmic divergence on the diagonal, that is
\begin{equation*}
K(x,y)=-\ln|x-y|+F(x,y)
\end{equation*}
with $F(x,y)$ smooth.

\subsection{Wick ordering and Hermite polynomials}\label{sec:WickHermite}
Let $(B_t)_{t\in\mathbb{R}_+}$ be the standard $1d$ Brownian motion. We consider the Wick ordering of $(B_t)^{2n}$, defined by
\begin{equation*}
:(B_t)^{2n}: = t^n P^{H}_{2n}\left(\frac{B_t}{\sqrt{t}}\right)
\end{equation*}
where $P^{H}_{2n}$ denotes the Hermite polynomial (normalized to have unitary leading coefficient) of degree $2n$. The Wick ordering procedure requires that the expectation vanishes, i.e.
\begin{equation*}
\mathbb{E}\left[P^{H}_{2n}\left(\frac{B_t}{\sqrt{t}}\right)\right]=0,\quad \forall t\geq 0.
\end{equation*}
It follows that the Itô derivative of $P^{H}_{2n}(B_t)$ with respect to the Brownian filtration has no drift term. The Wick ordering procedure provides a natural martingale parametrized by the time $t$.

\textbf{Notation.} In the following we absorb the renormalization in $t$ for $P^{H}_{2n}$ and write
\begin{equation*}
P_{2n}(B_t)\coloneqq t^n P^{H}_{2n}\left(\frac{B_t}{\sqrt{t}}\right).
\end{equation*}

\textbf{Example.} For $n=2$, the Wick ordering yields
\begin{equation*}
P_{4}(B_t)=B_t^4-6t B_t^2+3t^2
\end{equation*}
which can be equally written as
\begin{equation*}
P_{4}(B_t)=(B_t^2-3t)^2-6t^2
\end{equation*}
and $P_4$ is bounded from below by $-6t^2$. We also deduce that the envelope of the \textbf{zero-graph}
\begin{equation*}
\{(t,B_t)\in\mathbb{R}_{+}\times\mathbb{R}, P_{2n}\left(B_t\right)=0\}
\end{equation*}
is given by two symmetric branches
\begin{equation*}
\bigcup\limits_{t\in\mathbb{R}_+}\{(t,\sqrt{(3+\sqrt{6})t})\}\cup\{(t,-\sqrt{(3+\sqrt{6})t})\}\subset\mathbb{R}_{+}\times\mathbb{R}.
\end{equation*}

\textbf{General case.}  In general, by linear combination, we define the Wick ordered polynomial of $B_t$ for any real, unitary polynomial $R$ of even degree $2n\geq 4$:
\begin{equation*}
:R(B_t):=P_{R}(B_t).
\end{equation*}
More precisely, if
\begin{equation*}
R(X)=\sum\limits_{i=0}^{2n}a_{i}X^i
\end{equation*}
with $a_{2n}=1$, then we define the associated Wick ordered polynomial $P_R(X)$ by
\begin{equation*}
P_R(X)=\sum\limits_{i=0}^{2n}a_{i}t^{\frac{i}{2}}P^{H}_{i}\left(\frac{X}{\sqrt{t}}\right).
\end{equation*}
The martingale property of $P_{R}(B_t)$ with respect to the Brownian filtration is preserved by linear combination. The envelope of the graph of the zeros of $P_{R}$ is given explicitly by
\begin{equation*}
\bigcup\limits_{t\in\mathbb{R}_+}\{(t,f_{R}(t))\}\cup\{(t,-f_{R}(t))\}\subset\mathbb{R}_{+}\times\mathbb{R}
\end{equation*}
where the positive branch $f_{R}\geq 0$ can be explicitly calculated. The example above shows that when $R(X)=X^{4}$,
\begin{equation*}
f_{X^{4}}(t)=\sqrt{(3+\sqrt{6})t}.
\end{equation*}

The following facts are elementary.
\begin{prop}[Envelope of zeros]\label{Prop:ZeroEnvelope}
Let $R$ be a real, unitary polynomial of even degree $2n\geq 4$. The function $f_{R}$ satisfies the following:
\begin{enumerate}
	\item There exists some constant $A>0$ only depending on $n$ such that $f_{R}(t)\leq t+A$ for all $t\in\mathbb{R}_{+}$;
	\item For every $\epsilon>0$, there exists some constant $A'=A'(n,\epsilon)$ such that $f_{R}(t)\leq\epsilon t+ A'$ for all $t\in\mathbb{R}_{+}$.
\end{enumerate}
\end{prop}

We will also consider the value of $P_{R}$ on the line $\{t+A\}_{t\geq 0}$ for some constant $A$.
\begin{prop}[Values on cones]\label{Prop:ConeValues}
For large enough $A$, the function
\begin{equation*}
t\mapsto P_{R}(t+A)
\end{equation*}
satisfies the following properties:
\begin{enumerate}
	\item It is positive for $t\in\mathbb{R}_+$;
	\item It is strictly increasing in $t$ for $t\in\mathbb{R}_+$.
\end{enumerate}
\end{prop}

\subsection{Cut-off regularization}
Since the Gaussian Free Field $X$ only makes sense as a distribution, it is suitable to define the measure
\begin{equation}
V_{R}(\Lambda)=\int_{\Lambda}:R(X)(x):d^2x
\end{equation}
using a cut-off procedure. We need the following assumption:
\begin{prop}[Smooth white noise decomposition]\label{Prop:CutoffReg}
We choose a cut-off regularization $(X_{u})_{u\in\mathbb{R}_+}$ satisfying the following properties:
\begin{enumerate}
	\item The covariance kernel $K$ can be written in the form
	\begin{equation*}
	K(x,y)=\int_{0}^{\infty}Q_u(x,y)du
	\end{equation*}
	where for all $x\neq y$, the above integral is convergent; $Q_u$ is a bounded symmetric positive definite kernel for any $u$.
	\item Setting $K_t=\int_{0}^{t}Q_u du$, there exists a positive constant $C$ such that
	\begin{equation*}
	\left|K_t(x,y)-\left(t\wedge\ln_+\frac{1}{|x-y|}\right)\right|\leq C.
	\end{equation*}
	\item We have $\lim\limits_{x\to\infty}Q_u(x,x)=1$ with uniform convergence in $x\in\Lambda$.
	\item For all $0<\beta<2$,
	\begin{equation*}
	\int_{\Lambda^2}\int_{0}^{\infty}e^{\beta u}|Q_u(x,y)|d^2x d^2y du<\infty.
	\end{equation*}
\end{enumerate}
\end{prop}
It is proven in \cite[Section~4.2]{lacoin2018path} that the GFF $X$ on $\Lambda$ can be fitted into this assumption. We will thus work under this assumption in the following.

We define $(X_t(x))_{x\in\Lambda,t\geq 0}$ to be the jointly continuous process in $x$ and $t$ with covariance kernel
\begin{equation*}
\mathbb{E}\left[X_s(x)X_t(y)\right]=\int_{0}^{t\wedge s}Q_u(x,y)du.
\end{equation*}
According to the above assumption, given $x\in\Lambda$, the process $(X_t(x))_{t\geq 0}$ is very similar to a standard Brownian motion. We assume for readability in the following that
\begin{equation*}
K_t(x,x)=t
\end{equation*}
so that $(X_t(x))_{t\geq 0}$ is a standard Brownian motion.

\subsection{Quadratic variation of martingales}\label{sec:QuadVariation}
We have the following lemma in probability concerning the exponential martingale:
\begin{lemm}[Exponential martingale]\label{lem:ExpMartingale}
For any continuous local martingale $M$ and any $\lambda\in\mathbb{C}$, the process
\begin{equation*}
\exp\left(\lambda M_t-\frac{\lambda^2}{2}\left<M,M\right>_t\right)
\end{equation*}
is a local martingale. We also write $\left<M\right>_t$ for the quadratic variation $\left<M,M\right>_t$.
\end{lemm}

In particular, if $M_0=1$ and $\left<M\right>_\infty<\infty$, then $M$ is a $L^2$-bounded continuous martingale and we have for $\alpha>0$ the following inequality
\begin{equation*}
\limsup\limits_{t\to\infty}\mathbb{E}\left[\exp\left(-\alpha M_t\right)\right]<\infty
\end{equation*}
in such a way that the limit of $M_t$ displays Gaussian concentration.

We refer to the classical text book \cite{Revuz_1999} for these results.

\section{Proof of the main theorem}
We now prove Theorem~\ref{th:Main} using martingale methods.

\subsection{Preliminary notations}
Fix a real, unitary polynomial $R$ of even degree $2n\geq 4$. Hereafter we sometimes drop the dependence on $R$ where there is no ambiguity.

Borrowing notations from Proposition~\ref{Prop:ZeroEnvelope}, we consider the two-branched envelope
\begin{equation*}
E\coloneqq \{(t,u)\in\mathbb{R}_+\times\mathbb{R};|u|=f_{R}(t)\}.
\end{equation*}
The envelop $E$ depends on $R$: we write simply $E$ for readability.

We introduce a cut-off at level $\pm g(t)$ where
\begin{equation*}
g(t)=t+A
\end{equation*}
and $A\geq 0$ is a large constant chosen later: we require that $f_{R}(t)\leq g(t)$ for all $t\in\mathbb{R}_+$ and that Proposition~\ref{Prop:ConeValues} holds.

We consider also the cone $C$ with two symmetric branches:
\begin{equation*}
C\coloneqq \{(t,u)\in\mathbb{R}_+\times\mathbb{R};|u|=g(t)\}.
\end{equation*}
Geometrically, the envelope $E$ is in between the two branches of $C$.

Let us rewrite the main theorem with the notations in the preliminary. We define a regularization of the measure $V_{R}(\Lambda)$ using the smooth white noise decomposition Proposition~\ref{Prop:CutoffReg}:
\begin{equation*}
D_t=\int_{\Lambda}P_{R}(X_t(x))d^2x.
\end{equation*}
We prove in the following that uniformly in $t$, there exists some $C(A)$ such that for all $\alpha\in\mathbb{R}$,
\begin{equation*}
\mathbb{E}\left[e^{\alpha D_t}\right]\leq e^{C(A)\alpha^2}.
\end{equation*}

\subsection{Strategy of the proof}
One first calculates the quadratic variation of the martingale $D_t$ in view of Lemma~\ref{lem:ExpMartingale}. We have
\begin{equation*}
\left<D\right>_t\leq\int_{\Lambda^2\times[0,t]}\left|P'_{R}(X_u(x))P'_{R}(X_u(y))\right|Q_u(x,y)d^2x d^2y du.
\end{equation*}
If the graph $(t,X_t(x))$ stays (uniformaly in $t$ and in $x$) inside the cone $C$, then $|P'_{R}(X_u(x))|$ cannot take exceptionally high values and the quadratic variation $\left<D\right>_t$ is uniformaly bounded in $t$ (see Lemma~\ref{lem:LowValueContribution} below) and the $L^2$-theory of martingales applies. By Lemma~\ref{lem:ExpMartingale}, the limiting measure would display Gaussian concentration bound.

Almost surely this is not the case: the process $X_t(x)$ goes out of the cone $C$ and takes high values. We consider for every $x\in\Lambda$ the stopping time
\begin{equation*}
H^x\coloneqq\inf\{s\geq 0;(s,X_s(x))\in C\}.
\end{equation*}
As the zero-value envelope $E$ is inside the cone $C$, after time $H^{x}$ the process $P_{R}(X_t(x))$ at point $x$ stays positive until the next time it returns to $E$.

Introduce a sequence of stopping times (always with respect to a fixed $x\in\Lambda$):
\begin{equation}
\begin{aligned}
H^x_k\coloneqq&\inf\{s\geq L^x_{k-1};(x,X_s(x))\in C\},\\
L^x_k\coloneqq&\inf\{s\geq H^x_{k};(x,X_s(x))\in E\}.
\end{aligned}
\end{equation}
By convention, $L^x_0\equiv 0$. We can write as a decomposition of times depending on whether $X_t(x)$ takes low or high values,
\begin{equation*}
\begin{aligned}
[L^x,H^x]\coloneqq&\bigcup\limits_{k\in\mathbb{N}}[L^x_k,H^x_{k+1}],\\
[H^x,L^x]\coloneqq&\bigcup\limits_{k\in\mathbb{N}^{*}}[H^x_k,L^x_k].
\end{aligned}
\end{equation*}
It follows that for all $x\in\Lambda$, $[L^x,H^x]\cup[H^x,L^x]=\mathbb{R}_+$ almost surely.

We will take advantage of the positivity between the stopping times $[H^x,L^x]$. More precisely, on one hand the total contribution of $P_{R}(X_t(x))$ from intervals of the form $[L^x,H^x]$ is bounded in $L^2$ (since it takes values inside the cone $C$), on the other hand the contribution of $P_{R}(X_t(x))$ from intervals $[H^x,L^x]$ has constant positive sign.

We quantify this observation in the following way:
\begin{prop}[High value cut-off]\label{Prop:HighValueCutoff}
We consider the following decomposition. Let
\begin{equation*}
D_L(t)=\int_{\Lambda}\left(\int_{0}^{t}P'_{R}(X_s(x))\mathbf{1}_{\{s\in[L^x,H^x]\}}ds\right)d^2x
\end{equation*}
and
\begin{equation*}
D_H(t)=\int_{\Lambda}\left(\int_{0}^{t}P'_{R}(X_s(x))\mathbf{1}_{\{s\in[H^x,L^x]\}}ds\right)d^2x
\end{equation*}
in such a way that
\begin{equation*}
D_t=D_L(t)+D_H(t).
\end{equation*}
Then we have the following inequality
\begin{equation}\label{eq:DecompositionInequality}
D_t\geq D_L(t)-Q
\end{equation}
where $Q$ denotes the positive quantity
\begin{equation*}
Q\coloneqq\int_{\Lambda}\left(\sum\limits_{i=1}^{\infty}\mathbf{1}_{\{H^{x}_{i}<\infty\}}P_{R}(g(H^x_{i}))\right)d^2x.
\end{equation*}
Note that $D_L, D_H, Q$ depend on $R$ but we drop this dependence in the notation.
\end{prop}
\begin{proof}
Fix $x\in\Lambda$ and one can check the following claims:
\begin{itemize}
	\item If $k\in\mathbb{N}$ is such that $t\in[L^x_k,H^{x}_{k+1}]$, then
\begin{equation*}
\begin{aligned}
&\quad\int_{0}^{t}P'_{R}(X_s(x))\mathbf{1}_{\{s\in[L^x,H^x]\}}ds\\
&=\left(P_{R}(X_t(x))-P_{R}(X_{L^x_k}(x))\right)+\sum\limits_{i=0}^{k-1}(P_{R}(X_{H^x_{i+1}}(x))-P_{R}(X_{L^x_i}(x))).
\end{aligned}
\end{equation*}
This is because for every $l<k$, the increment of the process $X_t(x)$ on the interval $[L^x_l,H^{x}_{l+1}]$ contributes exactly to one term in the above summation.
	\item If now $k\in\mathbb{N}$ is such that $t\in[H^{x}_k,L^x_k]$, then we have $P_{R}(X_t(x))\geq 0$ and
\begin{equation*}
\int_{0}^{t}P'_{R}(X_s(x))\mathbf{1}_{\{s\in[L^x,H^x]\}}ds=\sum\limits_{i=0}^{k-1}(P_{R}(X_{H^x_{i+1}}(x))-P_{R}(X_{L^x_i}(x))).
\end{equation*}
\end{itemize}

Notice now that for all $i\in\mathbb{N}$, $P_{R}(X_{L^x_i}(x))=0$ by definition of the zero envelope $E$ and hitting times $L^{x}_i$. A similar argument as above shows that $D_H(t)\leq 0$ for all $t\in\mathbb{R}_+$, so that
\begin{equation*}
D_t\geq D_L(t).
\end{equation*}

To prove Equation~\eqref{eq:DecompositionInequality}, write the above in the following form:
\begin{equation*}
P_{R}(X_t(x))\geq \int_{0}^{t}P'_{R}(X_s(x))\mathbf{1}_{\{s\in[L^x,H^x]\}}ds-\sum\limits_{i=0}^{\infty}\mathbf{1}_{\{H^{x}_{i}<\infty\}}P_{R}(X_{H^x_{i+1}}(x)).
\end{equation*}
Equation~\eqref{eq:DecompositionInequality} follows by integrating over $x\in\Lambda$.
\end{proof}

Now the proof of the main theorem boils down to two estimates, of which the first one corresponds to the $L^2$ part, and the second one corresponds to the high-value part.
\begin{lemm}[Low value contribution]\label{lem:LowValueContribution}
$D_L(t)$ is an honest martingale that has bounded quadratic variation: it converges in $L^2$ and satisfies the Gaussian concentration bound
\begin{equation}
\exists C(A), \forall \alpha\in\mathbb{R}, \mathbb{E}\left[e^{\alpha D_L(\infty)}\right]\leq e^{C(A)\alpha^2}.
\end{equation}
\end{lemm}
\begin{lemm}[High value contribution]\label{lem:HighValueContribution}
The other quantity $Q$ in the decomposition also satisfies a Gaussian concentration bound:
\begin{equation}
\exists C(A), \forall \alpha\in\mathbb{R}, \mathbb{E}\left[e^{\alpha Q}\right]\leq e^{C(A)\alpha^2}.
\end{equation}
\end{lemm}
Combining these two lemmas, Theorem~\ref{th:Main} follows.

\subsection{Proofs of technical estimates}
We start by proving Lemma~\ref{lem:LowValueContribution}.
\begin{proof}[Proof of Lemma~\ref{lem:LowValueContribution}]
The fact that $D_L(t)$ is a martingale follows from construction. It suffices to show that $\left<D_L\right>_{\infty}$ is bounded from above by a constant: Gaussian concentration then follows by Lemma~\ref{lem:ExpMartingale}. The calculation goes as follows:
\begin{equation*}
\left<D_L\right>_t\leq\int_{\Lambda^2\times[0,t]}\left|P'_{R}(X_s(x))P'_{R}(X_s(y))\right|\mathbf{1}_{\{s\in[L^x,H^x]\cap[L^y,H^y]\}}Q_u(x,y)d^2x d^2y du.
\end{equation*}

Since $P'_{R}(X_s(x))$ is polynomial of degree $2n-1$, it has subexponential growth at infinity and the conditioning on $s$ implies that $|X_s(x)|\leq g(s)$. We bound the above by
\begin{align*}
\left<D_L\right>_t&\leq C\int_{\Lambda^2\times[0,t]}e^{\frac{1}{2}g(s)}e^{\frac{1}{2}g(s)}\mathbf{1}_{\{s\in[L^x,H^x]\cap[L^y,H^y]\}}Q_s(x,y)d^2x d^2y ds\\
&\leq C\int_{\Lambda^2\times[0,t]}e^{s}Q_s(x,y)d^2x d^2y ds
\end{align*}
for some constant $C=C(R)$. The last integral is finite by the last item of Proposition~\ref{Prop:CutoffReg}.
\end{proof}

\begin{proof}[Proof of Lemma~\ref{lem:HighValueContribution}]
Recall some preliminaries on Doob martingales. Define the positive quantity
\begin{equation*}
Q^x=\sum\limits_{i=1}^{\infty}\mathbf{1}_{\{H^{x}_{i}<\infty\}}P_{R}(g(H^x_i))
\end{equation*}
($Q^x$ depends on $R$ but we alleviate the notation) so that
\begin{equation*}
Q=\int_{\Lambda}Q^x d^2x.
\end{equation*}
\begin{lemm}[$\mathcal{L}^1$-boundedness]\label{lem:DoobFiniteness}
We have $\mathbb{E}[Q]<\infty$.
\end{lemm}
\begin{proof}
We bound $\mathbb{E}[Q^x]$ uniformly in $x\in\Lambda$: the claim follows from integrating over $\Lambda$.

Consider the following quantity:
\begin{equation*}
Q^{x,m}=\sum\limits_{i=1}^{\infty}\mathbf{1}_{\{H^{x}_{i}\in(m-1,m]\}}P_{R}(g(m)).
\end{equation*}
By Proposition~\ref{Prop:ConeValues}, choose $A$ large enough such that $P_{R}$ is strictly increasing on $\mathbb{R}_+$ and
\begin{equation*}
P_{R}(g(m))=\sup_{v\in(m-1,m]}P_{R}(g(v))
\end{equation*}
such that
\begin{equation*}
Q^x\leq\sum\limits_{m=1}^{\infty}Q^{x,m}.
\end{equation*}

We now prove a standard estimate
\begin{equation}\label{eq:HittingtimeEstimate}
\mathbb{E}[\#\{i:H^x_i\in(m-1,m]\}]\leq\frac{8}{\sqrt{2\pi m}}e^{-\frac{m}{2}}.
\end{equation}
Given this and that the polynomial $P_{R}$ has sub-exponential growth at infinity, i.e.
\begin{equation*}
P_{R}(g(m))\leq C(A)e^{\frac{m}{4}},
\end{equation*}
the result follows by summing over $m$ then integrating over $x$. 

Notice that, with $A\geq 1$, 
\begin{equation*}
\mathbb{P}[\exists i, H^x_i\in(m-1,m]]\leq\mathbb{P}\left[\sup\limits_{s\leq m}|B_s|\geq m\right]\leq\frac{4}{\sqrt{2\pi m}}e^{-\frac{m}{2}}
\end{equation*}
by a standard Gaussian tail estimate. Using the Markov property for the Brownian motion,
\begin{equation*}
\mathbb{P}[\#\{i:H_i\in(m-1,m]\geq k+1\}|\#\{i:H_i\in(m-1,m]\geq k\}]\leq\frac{1}{2}
\end{equation*}
and Equation~\eqref{eq:HittingtimeEstimate} follows from summing over $k$.
\end{proof}

The Doob martingale $Q^{x}_t$ is defined as
\begin{equation*}
Q^{x}_t=\mathbb{E}[Q^x|\mathcal{F}_t]
\end{equation*}
(recall that $\mathcal{F}_t=\sigma\{X_s,s\in[0,t]\}$) and since it is a martingale associated to the Brownian filtration $\{X_t(x)\}$, we can write
\begin{equation*}
dQ^x_t=A^x_t dX_t(x).
\end{equation*}
Then the bracket $\left<Q\right>_\infty$ can be written as
\begin{equation}\label{eq:Qbracket}
\left<Q\right>_\infty=\int_{\Lambda^2\times\mathbb{R}_+}A^x_u A^y_u Q_u(x,y)d^2x d^2y du.
\end{equation}

We now control $A^x_t$ uniformly in $x$, according to whether $t\in[H^x,L^x]$ or $t\in[L^x,H^x]$. In the following we drop the dependency on $x$ to alleviate the notations. More precisely, we prove that uniformly over all $t$, with some constant $C(A)$ independent of $x$,
\begin{equation}\label{eq:FirstCaseBound}
|A_t|\leq C(A)e^{t/2}.
\end{equation}
Lemma~\ref{lem:HighValueContribution} then follows from Equation~\eqref{eq:Qbracket} and Lemma~\ref{Prop:CutoffReg}, together with Lemma~\ref{lem:ExpMartingale}.

To prove Equation~\eqref{eq:FirstCaseBound}, we apply coupling techniques to the Brownian motion $X_t(x)$.

\subsubsection{First case: $t\in[H^x,L^x]$}
Suppose $t\in[H^x_k,L^x_k]$ for some $k\in\mathbb{N}$.

Let $\mathbb{P}_z$ be the law of a standard Brownian motion $(B_t)_{t\geq 0}$ starting at point $z$. By the strong Markov property of $X_t(x)$ as a Brownian motion $B_t$ (we drop the index $x$ afterwards), write $Q_t$ as
\begin{equation*}
Q_t=\sum\limits_{i=1}^{k-1}P_{R}(g(H_{i}))+\mathbb{E}_{X_t(x)}\left[\sum\limits_{i=1}^{\infty}\mathbf{1}_{\{\overline{H}^t_i<\infty\}}P_{R}(g(t+\overline{H}^t_i))\right]
\end{equation*}
where the stopping time sequence $\overline{H}^{t}_k=\overline{H}^{t}_k(B)$ is defined recursively by $\overline{H}^{t}_0=0$ and
\begin{equation}\label{eq:NewHittingtimes}
\begin{aligned}
\overline{L}^t_k&=\inf\{s\geq\overline{H}^t_{k-1}; B_s\in E\};\\
\overline{H}^t_k&=\inf\{s\geq\overline{L}^t_{k}; B_s\in C\}.
\end{aligned}
\end{equation}
We deduce the expression for $A_t$ in this case:
\begin{equation*}
A_t=\partial_z\left(\mathbb{E}_z\left[\sum\limits_{i=1}^{\infty}\mathbf{1}_{\{\overline{H}^t_i<\infty\}}P_{R}(g(t+\overline{H}^t_i))\right]\right)\big|_{z=X_t(x)}.
\end{equation*}
We show that the expression in the definition of $A_t$ that we derive is Lipschitz in $z$ with the adequate Lipschitz constant: this will imply Equation~\eqref{eq:FirstCaseBound}.

Let $t\in[H,L]$ and consider a coupling between two independent Brownian motions, starting from points $z_1<z_2$ with $|z_1-z_2|$ small, denoted respectively by $B^1$ and $B^2$. Suppose that the two Brownian motions evolve independently until the first time they meet
\begin{equation*}
\tau=\inf\{s>0; B^1_s=B^2_s\}
\end{equation*}
and jointly afterwards. Each Brownian motion in this coupling defines its own hitting time $\overline{H}^{(t,j)}_{i},\overline{L}^{(t,j)}_{i}$ for $j\in\{1,2\}$ similarly as in Equations~\eqref{eq:NewHittingtimes}. The hitting times are identical up to a shift in the indices after merging at time $\tau$.

If $\tau<\min\{\overline{H}^{(t,1)}_{1},\overline{H}^{(t,2)}_{1}\}$, then each Brownian motion gives rise to the same contribution in the expression of $A_t$. In particular, this also holds for $\tau<\min\{\overline{H}^{(t,1)}_{1},\overline{H}^{(t,2)}_{1},1\}$. Hereafter let
\begin{equation*}
\mathcal{T}=\min\{\overline{H}^{(t,1)}_{1},\overline{H}^{(t,2)}_{1},1\}.
\end{equation*}

It suffices to show that the following bound:
\begin{equation}
\begin{aligned}
&\quad\left|\mathbb{E}_{z_1}\left[\sum\limits_{i=1}^{\infty}\mathbf{1}_{\{\overline{H}^t_i<\infty\}}P_{R}(g(t+\overline{H}^t_i))\right]-\mathbb{E}_{z_2}\left[\sum\limits_{i=1}^{\infty}\mathbf{1}_{\{\overline{H}^t_i<\infty\}}P_{R}(g(t+\overline{H}^t_i))\right]\right|\\
&\leq\mathbb{P}\left[\tau>\mathcal{T}\right]\times\mathbb{E}\left[\sum\limits_{j=1,2}\sum\limits_{i=1}^{\infty}\mathbf{1}_{\{\overline{H}^{(t,j)}_i<\infty\}}P_{R}(g(t+\overline{H}^{(t,j)}_i))\big|\tau>\mathcal{T}\right].
\end{aligned}
\end{equation}

$\bullet$ One first shows by standard coupling estimate on Brownian motions that
\begin{equation*}\label{eq:CouplingEstimate}
\mathbb{P}\left[\tau>\min\{\overline{H}^{(t,1)}_{1},\overline{H}^{(t,2)}_{1},1\}\right]\leq C|z_1-z_2|.
\end{equation*}
It is a standard Brownian coupling result that $\mathbb{P}[\tau>1]\leq C|z_1-z_2|$. It remains to show
\begin{equation*}
\mathbb{P}\left[\tau>\overline{H}^{(t,1)}_{1}\right]\leq C|z_1-z_2|.
\end{equation*}
Provided that we choose a large enough $A$ in the definition of $g_t$, we have
\begin{equation*}
\overline{H}^{(t,1)}_{1}\geq\min\{s;|B^1_s-z_1|\geq 1\}
\end{equation*}
and
\begin{equation*}
\mathbb{P}\left[\tau>\overline{H}^{(t,1)}_{1}\right]\leq\mathbb{P}_{(0,z_2-z_1)}\left[\mathcal{T}_{\Delta}>\mathcal{T}_{\{-1,1\}\times\mathbb{R}}\right]\leq C|z_1-z_2|
\end{equation*}
where $\Delta=\{(x,x);x\in\mathbb{R}\}$ and $\mathcal{T}_{A}$ denotes the hitting time of a set $A$ by a two-dimensional Brownian motion. This is a standard estimate (for a detailed proof, \cite[Appendix~B]{lacoin2018path}).

$\bullet$ We now show that
\begin{equation*}
\mathbb{E}\left[\sum\limits_{j=1,2}\sum\limits_{i=1}^{\infty}\mathbf{1}_{\{\overline{H}^{(t,j)}_i<\infty\}}P_{R}(g(t+\overline{H}^{(t,j)}_i))\big|\tau>\min\{\overline{H}^{(t,1)}_{1},\overline{H}^{(t,2)}_{1},1\}\right]\leq Ce^{t/2}.
\end{equation*}

By linearity it suffices to show it for $j=1$, the calculation for $j=2$ is similar. We apply Markov property at $\mathcal{T}=\min\{\overline{H}^{(t,1)}_{1},\overline{H}^{(t,2)}_{1},1\}$ and distinguish two subcases:

-- If $\mathcal{T}<\overline{L}^{(t,1)}_1$, we apply Markov property for the Brownian motion $B^1_s$ at $\overline{L}^{(t,1)}_1$. Since at time $\overline{L}^{(t,1)}_1$ the Brownian motion $B^1$ takes value $f_{R}(t+\overline{L}^{(t,1)}_1)$, the above quantity is dominated by
\begin{equation*}
\sup_{r\in[t,t+1]}\sup_{s\geq r}\mathbb{E}_{f_{R}(s)}\left[\sum\limits_{i=1}^{\infty}\mathbf{1}_{\{\overline{H}^{(s,1)}_i<\infty\}}P_{R}(g(s+\overline{H}^{(s,1)}_i))\right].
\end{equation*}
As in Lemma~\ref{lem:DoobFiniteness}, together with Proposition~\ref{Prop:ConeValues} for any $r\in[t,t+1]$ and $s\geq r$,
\begin{equation}
\begin{aligned}\label{eq:f2nbounded}
&\quad\mathbb{E}_{f_{R}(s)}\left[\sum\limits_{i=1}^{\infty}\mathbf{1}_{\{\overline{H}^{(s,1)}_i<\infty\}}P_{R}(g(s+\overline{H}^{(s,1)}_i))\right]\\
&\leq C\sum\limits_{n\geq 1}\mathbb{E}_{f_{R}(s)}\left[\#\{\overline{H}^{(s,1)}_i\in(n-1,n]\}\right]P_{R}(g(s+n))\\
&\leq C\sum\limits_{n\geq 1}\frac{1}{\sqrt{n}}e^{-\frac{((1-\epsilon)s+n)^2}{2n}}e^{(s+n)/3}\\
&\leq Ce^{-s/3}.
\end{aligned}
\end{equation}
We used the Proposition~\ref{Prop:ConeValues} that one can assume $f_{R}(s)<\epsilon s+A'(\epsilon)$ uniformly for any $\epsilon>0$.

-- If $\mathcal{T}>\overline{L}^{(t,1)}_1$, then since $\mathcal{T}\leq\overline{H}^{(t,1)}_1$ by definition, we know that
\begin{equation*}
|B^1_{\mathcal{T}}|\leq g(t+\mathcal{T}).
\end{equation*}
By assumption, $\mathcal{T}\leq 1$ and we can control the contribution by
\begin{equation*}
\sup_{r\in[t,t+1]}\sup_{|z|\leq g(r)}\mathbb{E}_{z}\left[\sum\limits_{i=1}^{\infty}\mathbf{1}_{\{\overline{H}^{(r,1)}_i<\infty\}}P_{R}(g(r+\overline{H}^{(r,1)}_i))\right].
\end{equation*}
Again, by the same argument as in Lemma~\ref{lem:DoobFiniteness}, with $|z|\leq g(r)$ and Proposition~\ref{Prop:ConeValues},
\begin{equation*}
\begin{aligned}
&\quad\mathbb{E}_{z}\left[\sum\limits_{i=1}^{\infty}\mathbf{1}_{\{\overline{H}^{(r,1)}_i<\infty\}}P_{R}(g(r+\overline{H}^{(r,1)}_i))\right]\\
&\leq C\sum\limits_{n\geq 1}\mathbb{P}_z\left[\exists i; \overline{H}^{(r,1)}_i\in(n-1,n]\right]P_{R}(g(r+n))\\
&\leq C\sum\limits_{n\geq 1}\frac{1}{\sqrt{n}}e^{-n/2}e^{(r+n)/3}\\
&\leq Ce^{r/3}
\end{aligned}
\end{equation*}
so that the contribution above is control by $Ce^{t/3}$. This is an appropriate Lipschitz constant for Equation~\eqref{eq:FirstCaseBound}.

\subsubsection{Second case: $t\in[L^x,H^x]$}
In this case, the above strategy fails for the first term $i=1$. Indeed, if both Brownian motions start near the cone $C$, the probability that they merge before either of them hitting $C$ is arbitrarily small and Equation~\eqref{eq:CouplingEstimate} cannot be reproduced. However the same argument works for $i\geq 2$ (since in this case they both have to travel from the inner envelope $E$ to the outer cone $C$). We thus have to look more carefully into the term $i=1$.

For the $i=1$ case, we use a different ``parallel'' coupling. Consider two Brownian motions, $B^1_s$ starting at $z_1$ and $B^2_s$ starting at $z_2$ (by symmetry, suppose that $z_2>z_1$) coupled as
\begin{equation*}
B^2_s=B^1_s+(z_2-z_1).
\end{equation*}

Denote by $S_1$ (resp. $S_2$) the hitting time of $B^1$ (resp. $B^2$) at the outer cone $C$. We show that
\begin{equation}\label{eq:absolutebound}
\left|\mathbb{E}\left[\mathbf{1}_{\{S_1<\infty\}}P_{R}(g(t+S_1))\right]-\mathbb{E}\left[\mathbf{1}_{\{S_2<\infty\}}P_{R}(g(t+S_2))\right]\right|\leq C(A)|z_1-z_2|e^{t/2}.
\end{equation}

By symmetry we can add the indicator of the event that $S_1<S_2$ (otherwise change $(z_1,z_2)$ into $(-z_1,-z_2)$). This is only a geometric data: with the assumption that $z_1<z_2$, the event $S_1<S_2$ is equivalent to the event that the first time any of the Brownian motions $B^1_s$ and $B^2_s$ hits the the outer cone $C$, the location is at the lower branch of $C$.

Since the above inequality is an absolute value, we should seperate into two subcases:

$\bullet$ We can choose $A$ large enough so that $P_{R}(g(\cdot))$ is strictly increasing on $\mathbb{R}_+$ by Proposition~\ref{Prop:ConeValues}. Notice that with the conditioning $S_1<S_2$ and Markov property at $S_1$,
\begin{equation*}
\begin{split}
&\quad\mathbb{E}\left[\mathbf{1}_{\{S_2<\infty\}}P_{R}(g(t+S_2))\mathbf{1}_{\{S_1<S_2\}}\right]\\
&\geq\mathbb{E}\left[\mathbf{1}_{\{S_2<\infty\}}P_{R}(g(t+S_1))\mathbf{1}_{\{S_1<S_2\}}\right]\\
&\geq\mathbb{E}\left[\mathbf{1}_{\{S_1<\infty\}}P_{R}(g(t+S_1))\mathbf{1}_{\{S_1<S_2\}}\right]\inf_{S_1>0}\mathbb{P}\left[S_2<\infty|S_1<\infty\right]
\end{split}
\end{equation*}
and thus we have
\begin{equation*}
\begin{aligned}
&\quad\mathbb{E}\left[\mathbf{1}_{\{S_1<\infty\}}P_{R}(g(t+S_1))\mathbf{1}_{\{S_1<S_2\}}\right]-\mathbb{E}\left[\mathbf{1}_{\{S_2<\infty\}}P_{R}(g(t+S_2))\mathbf{1}_{\{S_1<S_2\}}\right]\\
&\leq\left(\sup_{S_1>0}\mathbb{P}\left[S_2=\infty|S_1<\infty\right]\right)\mathbb{E}\left[\mathbf{1}_{\{S_1<\infty\}}P_{R}(g(t+S_1))\mathbf{1}_{\{S_1<S_2\}}\right].
\end{aligned}
\end{equation*}

Repeating arguments as before, we have
\begin{equation*}\label{eq:z1estimate}
\mathbb{E}\left[\mathbf{1}_{\{S_1<\infty\}}P_{R}(g(t+S_1))\mathbf{1}_{\{S_1<S_2\}}\right]\leq\mathbb{E}\left[\mathbf{1}_{\{S_1<\infty\}}P_{R}(g(t+S_1))\right]\leq C(A)e^{t/2}.
\end{equation*}

It remains to control the other conditional probability. Using the strong Markov property for $B^2$ at $S_1$, for all $S_1>0$,
\begin{equation*}
\mathbb{P}\left[S_2<\infty|S_1<\infty\right]\geq\mathbb{P}_{z_1-z_2}\left[\exists s; B_s=s\right]=e^{-2(z_2-z_1)}
\end{equation*}
by standard diffusion process identity. Indeed, for $z\leq 0$,
\begin{equation*}
u(z)=\mathbb{P}_{z}\left[\exists s; B_s=s\right]
\end{equation*}
solves the differential equation
\begin{equation*}
u''(z)-2u'(z)=0
\end{equation*}
with initial condition $u(0)=1$ and $u(-\infty)=0$. Together this proves one direction in Equation~\eqref{eq:absolutebound}

$\bullet$ Now we prove the other direction. Let $\mathcal{T}=\min\{S_2,S'_2\}$ with
\begin{equation*}
S'_2=\inf\{s>S_1;B^2_s=-f_{R}(t+s)\}.
\end{equation*}
Only the lower branch of the envelope $E$ is concerned because the assumptions $z_1<z_2$ and $S_1<S_2$ imply geometrically that at time $S_1$, the Brownian motion $B^1_{S_1}$ is located at the lower part of the cone $C$ and the other coupled Brownian motion $B^2_{S_1}$, at time $S_1$, is between the lower part of $E$ and the lower part of $C$. It follows that after time $S_1$, the Brownian motion $B^2$ first hits either the lower part of $E$ (corresponding to $S'_2$) or the lower part of $C$ (corresponding to $S_2$).


We give bounds on $\mathbb{E}\left[\mathbf{1}_{\{S_2<\infty\}}P_{R}(g(t+S_2))\mathbf{1}_{\{S_1<S_2\}}\right]$ depending on how $S_2$ compares to $S'_2$. We are going to show that in one case
\begin{equation*}
\mathbb{E}\left[\mathbf{1}_{\{S_2<\infty\}}P_{R}(g(t+S_2))\mathbf{1}_{\{S_1<S_2\}}\mathbf{1}_{\{S_2<S'_2<\infty\}}\right]\leq \mathbb{E}\left[\mathbf{1}_{\{S_1<\infty\}}P_{R}(g(t+S_1))\mathbf{1}_{\{S_1<S_2\}}\right]
\end{equation*}
and in the other case
\begin{equation*}
\mathbb{E}\left[\mathbf{1}_{\{S_2<\infty\}}P_{R}(g(t+S_2))\mathbf{1}_{\{S_1<S_2\}}\mathbf{1}_{\{S'_2<S_2<\infty\}}\right]\leq C|z_1-z_2|.
\end{equation*}
The sum of these two equalities yields a constant order Lipschitz coefficient for Equation~\eqref{eq:FirstCaseBound}.

-- For the first case, consider the Brownian motion
\begin{equation*}
\overline{B}_s=B^2_{S_1+s}-B^2_{S_1}.
\end{equation*}
Now in this case, $\mathcal{T}-S_1$ is a stopping time for $\overline{B}$. Since conditioned on the event that $S_1<\infty$ and $S_1<S_2$, $\{P_{R}(B^{2}_{S_1}+\overline{B}_u)\}_{u\geq 0}$ is a (positive) martinagle for the filtration
\begin{equation*}
\overline{\mathcal{F}}_u=\overline{\mathcal{F}}_{S_1}\cup\sigma(\overline{B}_s;s\leq u),
\end{equation*}
up until time $(\mathcal{T}-S_1)$. Fatou's lemma for the conditional expectation yields (since $P_{R}(f_{R}(\cdot))=0$)
\begin{equation*}
\begin{aligned}
&\quad\mathbb{E}\left[\mathbf{1}_{\{S_2<\infty\}}P_{R}(g(t+S_2))\mathbf{1}_{\{S_1<S_2\}}\mathbf{1}_{\{S_2<S'_2<\infty\}}\right]\\
&\leq\mathbb{E}\left[\mathbf{1}_{\{S_1<\infty\}}P_{R}(B^{2}_{S_1})\mathbf{1}_{\{S_1<S_2\}}\right]\\
&\leq\mathbb{E}\left[\mathbf{1}_{\{S_1<\infty\}}P_{R}(g(t+S_1))\mathbf{1}_{\{S_1<S_2\}}\right]
\end{aligned}
\end{equation*}
provided that $P_{R}(B^{2}_{S_1})\leq P_{R}(g(t+S_1))$ by Proposition~\ref{Prop:ConeValues}.

-- For the second case, consider
\begin{equation*}
\mathbb{E}\left[P_{R}(g(t+S_2))\mathbf{1}_{\{S_1<S'_2<S_2<\infty\}}\right].
\end{equation*}
By applying Markov property at time $S'_2$,
\begin{equation*}
\mathbb{E}\left[P_{R}(g(t+S_2))\mathbf{1}_{\{S_1<S'_2<S_2<\infty\}}\right]\leq\mathbb{P}\left[S'_2<S_2<\infty\right]\max_{s\geq t}\mathbb{E}_{f_{R}(s)}\left[P_{R}(g(H^s_1))\mathbf{1}_{\{H^{s}_1<\infty\}}\right].
\end{equation*}

The second term on the right hand side is bounded by a constant, see Equation~\eqref{eq:f2nbounded}. The rest reduces to the estimate (by applying Markov property at time $S'_2$)
\begin{equation*}
\begin{aligned}
&\quad\mathbb{P}\left[S'_2<S_2|S_1<\infty\right]\\
&=\mathbb{P}_{z_1-z_2}\left[\min\{s;B_s=f_{R}(s+t)-t-A\}<\min\{s;B_s=s\}\right]\\
&\leq\mathbb{P}_{z_1-z_2}\left[\min\{s;B_s=\epsilon(s+t)+A'(\epsilon)-t-A\}<\min\{s;B_s=s\}\right]
\end{aligned}
\end{equation*}
where we used the fact that $f_{R}(s)\leq \epsilon s+A'(\epsilon)$ for arbitrarily small $\epsilon>0$. The last probability can be shown to be smaller than $C|z_1-z_2|$ by diffusion process estimate. Indeed, it can be bounded by
\begin{equation*}
u(z_1-z_2)=\mathbb{P}_{z_1-z_2}\left[\min\{s;B_s=s+A'-A\}<\min\{s;B_s=s\}\right]
\end{equation*}
where the last term solves the differential equation
\begin{equation*}
u''(z)-2u'(z)=0
\end{equation*}
with initial conditions $u(0)=0$, $u(A'-A)=1$. Computation yields
\begin{equation*}
u(z)=\frac{1-e^{-2z}}{1-e^{-2(A'-A)}}\leq C|z|
\end{equation*}
which completes the proof.
\end{proof}

\bibliographystyle{alpha}
\bibliography{reference}

\begin{thebibliography}{LRV19}

\bibitem[Dub09]{dubedat2009sle}
Julien Dub{\'e}dat.
\newblock {SLE} and the free field: partition functions and couplings.
\newblock {\em Journal of the American Mathematical Society}, 22(4):995--1054,
  2009.

\bibitem[LRV18]{lacoin2018path}
Hubert Lacoin, R{\'e}mi Rhodes, and Vincent Vargas.
\newblock {Path integral for quantum Mabuchi K-energy}.
\newblock {\em arXiv preprint arXiv:1807.01758}, 2018.

\bibitem[LRV19]{lacoin2019probabilistic}
Hubert Lacoin, R{\'e}mi Rhodes, and Vincent Vargas.
\newblock A probabilistic approach of ultraviolet renormalisation in the
  boundary {Sine-Gordon} model.
\newblock {\em arXiv preprint arXiv:1903.01394}, 2019.

\bibitem[Nel66]{nelson1966quartic}
Edward Nelson.
\newblock A quartic interaction in two dimensions.
\newblock In {\em Mathematical Theory of Elementary Particles, Proc. Conf.,
  Dedham, Mass., 1965}, pages 69--73. MIT Press, 1966.

\bibitem[Ree12]{reed2012methods}
Michael Reed.
\newblock {\em Methods of modern mathematical physics: Functional analysis}.
\newblock Elsevier, 2012.

\bibitem[RY99]{Revuz_1999}
Daniel Revuz and Marc Yor.
\newblock {Continuous Martingales and Brownian Motion}.
\newblock {\em Grundlehren der mathematischen Wissenschaften}, 1999.

\bibitem[Sim15]{simon2015p}
Barry Simon.
\newblock {\em {The P(Phi)2 Euclidean (Quantum) Field Theory}}.
\newblock Princeton University Press, 2015.

\end{thebibliography}

\end{document}